\documentclass[11pt,a4paper]{amsart}
\usepackage[utf8]{inputenc}
\usepackage{amsmath,amsfonts}
\usepackage{graphicx}
\usepackage{subcaption}

\advance \textheight by 0,18mm 

\usepackage{hyperref}
\usepackage{xurl}

\usepackage{tikz}
\usetikzlibrary{positioning,arrows.meta}

\newtheorem{lemma}{Lemma}
\newtheorem{lemmaNOTcounted}{Lemma}

\newtheorem{corollaryQ}{Corollary}
\newcommand{\corollaryQold}{}\let\corollaryQold=\thecorollaryQ
\def\thecorollaryQ{Q\corollaryQold}

\newtheorem{corollaryX}{Corollary}
\newcommand{\corollaryXold}{}\let\corollaryXold=\thecorollaryX
\def\thecorollaryX{X\corollaryXold}

\newtheorem{theorem}[lemma]{Theorem}
\newtheorem{proposition}[lemma]{Proposition}
\newtheorem{problem}{Problem}

\begin{document}
\title
[Generalized Combinatorial 
Nullstellensatz Lemma]
{The Generalized  Combinatorial Laso\'n--Alon--Zippel--Schwartz Nullstellensatz
  Lemma}
\author{G\"unter Rote}
\begin{abstract}
  We survey
a few 
  strengthenings and 
  generalizations of
the Combinatorial Nullstellensatz of Alon
and
the Schwartz--Zippel Lemma.
These lemmas guarantee the existence of (a certain number of) nonzeros
of a multivariate polynomial
when the variables run independently through sufficiently large ranges.
\end{abstract}

\maketitle

\tableofcontents
\section{Introduction}

\subsection{The Quantitative and the Existence Conclusion}

Consider  a polynomial $f\in K[x_1,\ldots,x_n]$ in $n$ variables over a
field or integral domain $K$, and let $S_1,\dots,S_n$ be subsets of
$K$.
We want to make statements about the nonzeros of $f(x_1,\dots,x_n)$
when the variables $x_i$ run independently over the sets $S_i$, under the assumption that
these sets are sufficiently large, compared to certain parameters
$d_1,\dots,d_n$ that are related to the degrees of the terms in~$f$.
We may then derive a mere conclusion about the \emph{existence} of a nonzero
or a stronger statement about the \emph{number} of nonzeros:

\begin{quote}
  \textsc{The Quantitative Conclusion.}
  If $|S_i|> d_i$ for all $i=1,\ldots,n$, then
the number of
tuples $(x_1,\dots,x_n) \in
S_1\times S_2 \times \dots \times S_n$
 such that
$f(x_1,\dots,x_n) \ne 0$ is
  \emph{at least}
%
%
  \begin{multline}\label{bound-equal}
    \qquad 
(|S_1|-d_1)
\cdot
(|S_2|-d_2)
\cdots
(|S_n|-d_n)
\\= 
|S_1\times S_2 \times \dots \times S_n| \cdot
\bigl(1-\tfrac{d_1}{|S_1|}\bigr)
\bigl(1-\tfrac{d_2}{|S_2|}\bigr)
\cdots
\bigl(1-\tfrac{d_n}{|S_n|}\bigr).
\quad
\end{multline}
\end{quote}
The product in the right half
of the last line 
can be interpreted as
a lower bound on the \emph{probability} of getting a nonzero.

Since the product of the terms $|S_i|-d_i$ is positive, an immediate
consequence is
\begin{quote}
  \textsc{The Existence Conclusion.}
If $|S_i|> d_i$ for all $i=1,\ldots,n$, then {there exists}
a tuple of values $(x_1,\dots,x_n) \in S_1\times S_2 \times \dots \times
S_n$ such that
$f(x_1,\dots,x_n) \ne 0$.
\end{quote}

\subsection{Assumptions on the numbers \texorpdfstring{$d_i$}{d(i)}}
\label{sec:assump}

These conclusions hold under a variety of different
\emph{assumptions} about the parameters $d_1,\dots,d_n$.

To describe these parameters,
we recall a few standard definitions.
A \emph{monomial}
is a product  $x_1^{a_1}x_2^{a_2}\ldots x_n^{a_n}$ of  powers of variables $x_i$ (not
including a coefficient from $K$).
The degree of the monomial \emph{in the variable $x_i$} is the
exponent $a_i$,
and the \emph{total degree} is the sum
 $a_1+\cdots+a_n$
of these exponents.
The \emph{monomials of a polynomial} $f$
 are the monomials that have
nonzero coefficients when the polynomial is written out in expanded form as a linear
combination of monomials. 

The (partial) degree of a polynomial $f$ in the variable $x_i$ 
(or the degree of $x_i$ in $f$) is the largest exponent $a_i$ for
which
$x_i^{a_i}$ appears as a factor of a monomial of~$f$.
%
The total degree of a polynomial is the largest total degree of
any of its monomials. This is what is usually called \emph{the degree}
of the polynomial without further qualification.

A monomial of $f$ is \emph{maximal} if it does not divide another
monomial of~$f$, see Figure~\ref{fig-maximal}.

{
\def\thelemmaNOTcounted{X}

\begin{lemmaNOTcounted}
  [
  Generalized Combinatorial Nullstellensatz,
  Laso\'n 2010~{\cite[Theorem 2]{l-gcn-10}},
  Tao and Vu 2006~{\cite[Exercise~9.1.4, p.~332]{tv-ac-06}}
  ]
\label{stark}
\label{main-existence}
If $x_1^{d_1}x_2^{d_2}\ldots x_n^{d_n}$ is a maximal
monomial of $f$, then the Existence Conclusion holds.
\end{lemmaNOTcounted}

}

The
\emph{lexicographically largest}
 monomial
 $x_1^{a_1}x_2^{a_2}\ldots x_n^{a_n}$
 of $f$ is defined in the usual sense,
 see Figure~\ref{fig-lex}:
 $a_1$ is the largest exponent of $x_1$
in all monomials of $f$,
 $a_2$~is the largest exponent of $x_2$
in all monomials that contain $x_1^{a_1}$ as a factor,
 $a_3$~is the largest exponent of $x_3$
in all monomials that contain $x_1^{a_1}x_2^{a_2}$ as a factor,
and so on.
Of course, we may get a different
lexicographically largest monomial if we consider the variables in a
different order. The results remain valid independently
of the chosen order.

\begin{figure}[tb]
  \centering
  \begin{subfigure}[c]{0.36\linewidth}
    \centering
  \includegraphics[page=5]{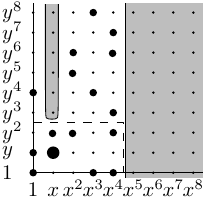}
    \subcaption{successively largest}
    \label{succ-largest}
  \end{subfigure}
$\to$
  \begin{subfigure}[c]{0.36\linewidth}
    \centering
  \includegraphics[page=4]{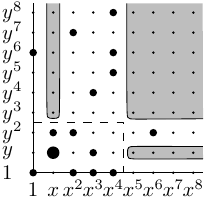}
    \subcaption{$d$-leading}
    \label{fig-d-leading}
  \end{subfigure}

\medskip
  \begin{subfigure}[c]{0.36\linewidth}
    \centering $\uparrow$
  \end{subfigure}
  \qquad
  \begin{subfigure}[c]{0.36\linewidth}
    \centering $\uparrow$
  \end{subfigure}
\bigskip

  \begin{subfigure}[c]{0.36\linewidth}
    \centering
    \includegraphics[page=2]{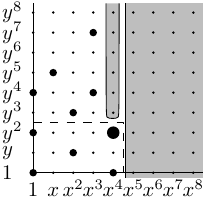}
    \subcaption{lexicographically
      largest}
    \label{fig-lex}
  \end{subfigure}
$\to$
  \begin{subfigure}[c]{0.36\linewidth}
    \centering
  \includegraphics[page=3]{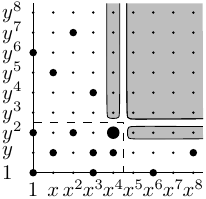}
    \subcaption{maximal}
    \label{fig-maximal}
  \end{subfigure}\hfill

\medskip
  \begin{subfigure}[c]{0.36\linewidth}
    \centering $\uparrow$
  \end{subfigure}
  \qquad
  \begin{subfigure}[c]{0.36\linewidth}
    \centering $\uparrow$
  \end{subfigure}
\bigskip

  \begin{subfigure}[c]{0.36\linewidth}
    \centering
  \includegraphics[page=7]{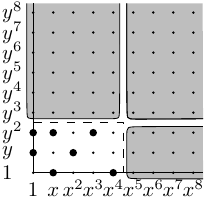}
    \subcaption{partial degrees}
    \label{fig-partial-degree}
  \end{subfigure}
$\to$
  \begin{subfigure}[c]{0.36\linewidth}
    \centering
  \includegraphics[page=6]{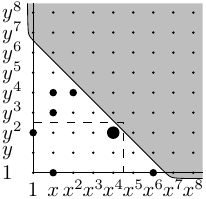}
    \subcaption{total degree}
    \label{fig-total-degree}
  \end{subfigure}
  \caption{The forbidden monomials for the various assumptions are
    shown as grey regions, for $(d_1,d_2)=(4,2)$.
In the top row, $(e_1,e_2)=(1,1)$ was chosen.}
  \label{fig:conditions}
\end{figure}

{
\def\thelemmaNOTcounted{Q}

  \begin{lemmaNOTcounted}
    \label{main-quantitative}
If the lexicographically largest
monomial of $f$ is
 $x_1^{d_1}x_2^{d_2}\ldots x_n^{d_n}$,
 then the Quantitative Conclusion holds. 
\end{lemmaNOTcounted}

}

\subsection{Applications}
\label{sec:importance}

Lemmas~\ref{main-quantitative}
and~\ref{main-existence}
and their many relatives in the literature (to be discussed shortly)
have numerous important applications to combinatorics and algorithms.
The results with the Quantitative Conclusion
are the basis for many randomized algorithms.
The prime example is
polynomial identity testing: Here one wants to check whether two
polynomials are identical, or whether a given
polynomial is identically zero.
The polynomials are given by some algorithm that can evaluate them for
specific values.
Lemmas~\ref{main-quantitative} provides a randomized test for this
property, provided some a-priori bounds on the degree can be given.
For 
more applications, see for example
\cite[Section~7]{mr-ra-95}. 

When applying the results with the Existence Conclusion, in particular the
Combinatorial Nullstellensatz (Corollary~\ref{alon}),
a nonzero
solution of the polynomial at hand represents some combinatorial
object
whose existence should be guaranteed.
See Alon \cite{a-cn-99} for a selection of applications.

The two application scenarios focus on different ends of the
probability spectrum.  In randomized algorithms, the ``success
probability'' of finding a nonzero should ideally be close to 1, but a reasonable
probability that decays only polynomially to zero is good enough. Then, by
choosing larger sets $S_i$ or by repeating the experiment, the
success probability can be amplified 
to any desired 
level.  The
precise 
probability bounds are not so important in this context.

On the other hand, when it comes to questions of existence,
the success of the argument comes down to whether the probability of
having a non-zero is non-zero or not. Here it is important to know the
smallest values $d_i$ for which the Existence Conclusion holds.

\subsection{
  Assumptions about the coefficient ring}
\label{sec:weaker-algebra}

To a lesser extent, the various results in the literature differ in
the assumption about the underlying ring of coefficients.
All results that we state (with the exception of
Lemmas~\ref{interpol} and~\ref{interpol1}
in Appendix~\ref{proof-tao-vu}, which require $K$ to be a field)
hold when $K$ is an integral domain, i.e., a
commutative ring without zero divisors.
We mention an even weaker condition under which the theorems hold:
$K$ can be an arbitrary commutative ring, but none of the differences $x-y$ for
$x,y\in S_i$ must be a zero divisor, see
 \cite[Definition~2.8]{schauz-08} or
 \cite[Condition~(D)]{bishnoi_2018}.
 

\begin{figure}[htb]
  \centering

\def\BOX#1#2#3#4#5#6{
  \node[boxnode,fill=#1,align=center] (#2) #3
    {\raise5pt\hbox{\strut}\strut\ #4\ \null\\[4pt]\hfill\footnotesize#5\null};
\node[above=-1.4mm of #2.north east]{\llap{\tiny \strut #6}};}

\def\BoxQ#1#2#3#4{
  \node[boxnode,fill=blue!10,align=center] (#1) #2
    {\raise5pt\hbox{\strut}\strut\ #3\ \null\\[4pt]\hfill\footnotesize#4\null};
\node[above=-1.4mm of #1.north east]{\llap{\tiny \strut Qu, Ex}};}

\def\BoxQ#1#2#3#4{\BOX{blue!10}{#1}{#2}{#3}{#4}{Qu, Ex}}
\def\BoxX#1#2#3#4{\BOX{red!10}{#1}{#2}{#3}{#4}{Ex}}
\def\BoxxQ#1#2#3#4{\BOX{blue!10,double}{#1}{#2}{#3}{#4}{Qu, Ex}}
\def\BoxxX#1#2#3#4{\BOX{red!10,double}{#1}{#2}{#3}{#4}{Ex}}

\begin{tikzpicture}[
  boxnode/.style={rectangle, draw=black, minimum size=6mm},
  style={->,shorten >=0.3pt,>={Stealth[round]},semithick}
]

\BoxxQ{lexmax}{}{lexicographically
  largest}{{Lemma~\ref{main-quantitative},
    Corollary~\ref{Q1}%
    \vspace{-2.2pt}}}
\BoxQ{sucmax}{[above=of lexmax]}{successively largest}{Theorem~\ref{theorem-knuth}}

\BoxxX{maximal}{[right=of lexmax]}{maximal}{Lemma~\ref{main-existence}}
\BoxX{totdeg}{[below=of maximal]}{largest total degree}{Corollar\smash
  y~\ref{alon}}
\BoxX{d-leading}{[above=of maximal]}{$d$-leading}{Theorem~\ref{schauz}}
\BoxQ{di-xi}{[below=of lexmax]}{$d_i$ = degree in $x_i$}
    {Corollar\smash y~\ref{generalized-demillo-lipton-zippel}}
\BoxQ{di-global}{[below=of di-xi]}{$d \ge{}$degree in $x_i$}{Corollar\smash y~\ref{zippel}}
\BoxQ{deMillo}{[below=of di-global]}{total degree $d$}{Corollar\smash y~\ref{deMillo}}
\BoxQ{Alon-Fueredi}{[left=of deMillo.171]}{total degree $d$,\ \null\\\null\ $d_i \ge{}$degree in
  $x_i$}{Theorem~\ref{gen-AF}}

\def\BoxQ#1#2#3#4{\BOX{blue!10}{#1}{#2}{#3}{#4}{Qu, Ex\quad}}
\BoxQ{totdegd}{[left=of di-xi]}{\,total degree $d$\,}{Corollar\smash y~\ref{schwartz-zippel}}
\node[below=-0.3mm of totdeg.south]{{\tiny \strut Combinatorial Nullstellensatz}};
\node[below=-0.3mm of totdegd.south]{{\tiny \strut Schwartz--Zippel Lemma}};


\draw[<-] (sucmax) -- (lexmax);
\draw[->] (sucmax) -- (d-leading);
\draw[<-] (maximal) -- (totdeg);
\draw[<-] (lexmax) -- (totdegd);
\draw[->] (lexmax) -- (maximal);
\draw[<-] (d-leading) -- (maximal);
\draw[<-] (lexmax) -- (di-xi);
\draw[<-] (totdeg) -- (di-xi);
\draw[<-] (di-xi) -- (di-global);
\draw[<-] (di-global) -- (deMillo);
\draw[->] (Alon-Fueredi) -- (deMillo);
\draw[->] (Alon-Fueredi.north) -- (di-xi);


\end{tikzpicture}
 
  \caption{Relation between the assumptions on $d_1,\ldots,d_n$. The
    Existence
    and/or some
    Quantitative
    Conclusion is indicated at the upper right corner of
    each box.}
  \label{fig:assumptions}
\end{figure}
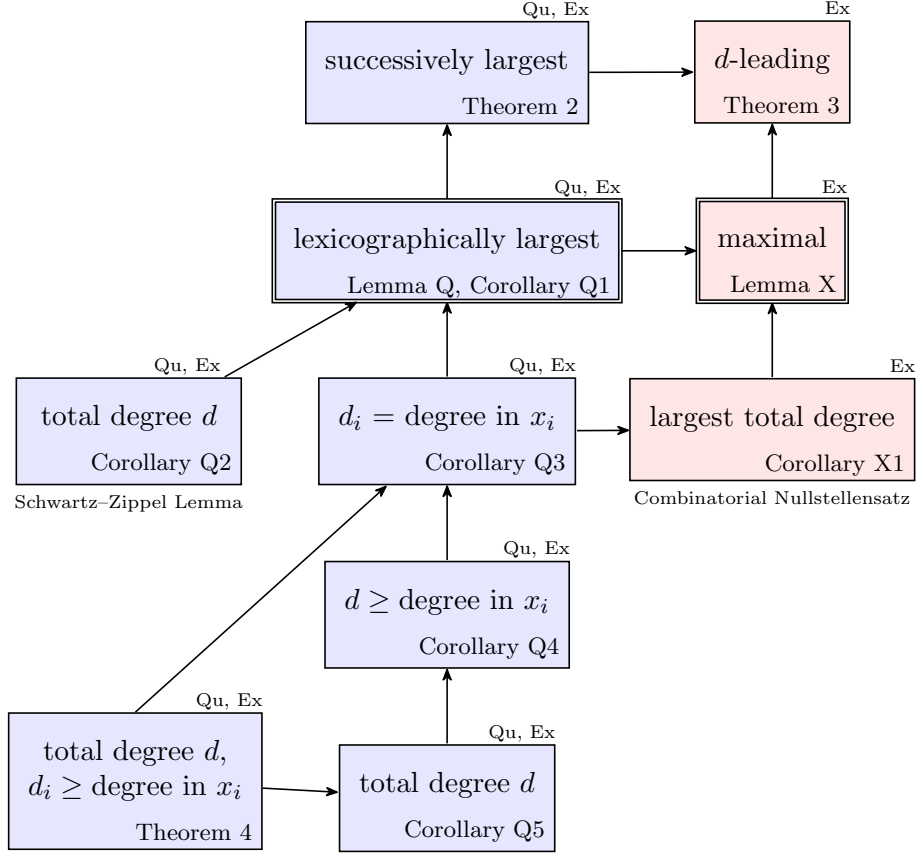

\subsection{Comparison of the assumptions}

Figure~\ref{fig:assumptions} compares the strength of the various
assumptions in these theorems,
including some conditions that are defined in later sections.

The
\emph{lexicographically largest} condition of
Lemma~\ref{main-quantitative}
implies
the 
\emph{maximality} assumption of
Lemma~\ref{main-existence},
but since
the Quantitative Conclusion in
Lemma~\ref{main-quantitative} is stronger than
the Existence Conclusion in
Lemma~\ref{main-existence},
neither of the two results can be derived from the other.
We will see in Section~\ref{sec:not} that
there is no common generalization.

While maximality is not sufficient to imply the Quantitative Conclusion,
there are some weaker quantitative conclusions that one can
derive under the maximality assumption, see
Section~\ref{sec:weak-conclusion}.

 The assumptions in Lemmas~\ref{main-existence} and \ref{main-quantitative} for the Existence or the
Quantitative Conclusion are not the weakest assumptions in terms of the monomials of $f$
that we are aware of.
The two boxes in the top row of
Figure~\ref{fig:conditions} and~\ref{fig:assumptions}
correspond to some weakened assumptions, which we treat in Section~\ref{weaker}.


\subsection{Tightness}
\label{sec:tight}

A simple family of polynomials shows that the bounds
of Lemmas~\ref{main-existence} and \ref{main-quantitative} are tight:
Select subsets $A_i\subset S_i$ of size $|A_i|=d_i$. Then the polynomial
\begin{equation}
  \label{eq:tight}
  \prod_{i=1}^n \prod_{a\in A_i}(x_i-a)
\end{equation}
has degree $d_i$ in each variable $x_i$.
It has
$(|S_1|-d_1)(|S_2|-d_2)\dots(|S_n|-d_n)$ zeros.
The term
$x_1^{d_1}x_2^{d_2}\ldots x_n^{d_n}$ is simultaneously the lexicographically
largest monomial and the unique maximal monomial, (and also the unique successively
largest exponent sequence
in the sense of Theorem~\ref{theorem-knuth} in
Section~\ref{weaker-quantitative}).

\subsection{Existence conclusions in the literature}
\label{sec:e}

This is Alon's original Combinatorial Nullstellensatz:

\begin{corollaryX}[{Combinatorial Nullstellensatz, Alon 1999 \cite[Theorem 1.2]{a-cn-99}}]
\label{alon}
If
$x_1^{d_1}x_2^{d_2}\ldots x_n^{d_n}$ is a monomial of largest total
degree, then the Existence Conclusion holds. 
\end{corollaryX}
Alon derives Corollary~\ref{alon} from a companion result,
\cite[Theorem 1.1]{a-cn-99}
(which can be proved by the trimming procedure of
Proposition~\ref{proposition-trimming}
in Section~\ref{sec:trimming}).
It states that, if the  Existence Conclusion does not hold, and $f$ is zero
on
$S_1\times S_2 \times \dots \times S_n$, it
can be represented in a certain way in the ideal
generated by the polynomials
$
\prod_{a\in S_i}(x_i-a)$.
This statement is analogous to Hilbert's Nullstellensatz, and this justifies the name
Combinatorial Nullstellensatz that Alon coined for these theorems.
It is of interest in its own right,
see
\cite[Section~9]{a-cn-99}
or \cite{clark-14},
but
we will not pursue these connections.

\subsection{Quantitative conclusions in the literature}
\label{sec:q}

The following bound follows by estimating the product
$(1-p_1)(1-p_2)\ldots(1-p_n)$ in \eqref{bound-equal} by the lower bound $1-p_1-p_2-\cdots-p_n$.

\begin{corollaryQ}[Schwartz 1979
{
\cite[Lemma 1]{s-pavpi-79,s-fpavp-80}}]
\label{Q1}
  Under the assumptions of Lemma~\ref{main-quantitative},
  i.e.,
  if the lexicographically largest
monomial of $f$ is
 $x_1^{d_1}x_2^{d_2}\ldots x_n^{d_n}$,
the number of nonzeros is at least
$$
|S_1\times S_2 \times \dots \times S_n| \cdot
\bigl(1-\tfrac{d_1}{|S_1|}-\tfrac{d_2}{|S_2|}-
\cdots
-\tfrac{d_n}{|S_n|}\bigr)
.
$$
\end{corollaryQ}
As a special case, when all sets $S_i$ are equal, we get
\begin{corollaryQ}[The 
  Schwartz--Zippel Lemma\footnote
  {see also 
    Wikipedia, \url{http://en.wikipedia.org/wiki/Schwartz-Zippel_lemma},
    accessed 2022-01-16},
  Schwartz 1979 
  {\cite[Corollary 1]{s-pavpi-79,s-fpavp-80}},
see also {\cite[Theorem~7.2]{mr-ra-95}} or
  {\cite[Exercise~9.1.1, pp.~331--332]{tv-ac-06}}]
  \label{schwartz-zippel}
  \ \\
If
$S_1=S_2=\dots=S_n=S$ and the polynomial has total degree $d\ge0$,
then
the number of nonzeros is at least
$$
|S|^n\cdot
\bigl(1-
\tfrac{d}{|S|}\bigr).
$$
In other words, the probability of getting a zero of $f$ if the
variables $x_i$ are uniformly and independently chosen from $S$ is at
most
$$d/|S|.$$
\end{corollaryQ}
The probabilistic formulation with the
upper bound ${d/|S|}$ on the probability of getting a zero is
the common statement of this lemma.
The same holds for the following statements, but for comparison, we
formulate all  theorems in terms of the number of nonzeros.

\medskip

The following statement looks at the degree of $f$ in each variable
$x_i$. It follows trivially from
Lemma~\ref{main-quantitative}.
\begin{corollaryQ}[Generalized DeMillo--Lipton--Zippel Theorem
  {\cite[Thm.~4.6]{bishnoi_2018}},
  \smallskip
  Knuth 
  1997 {\cite[Ex.~4.6.1--16, p.~436]{knuth2}}]
  \label{generalized-demillo-lipton-zippel}
\noindent\\
  If $d_i$ is the degree of variable $x_i$ in $f$,
  the Quantitative Conclusion holds.
\end{corollaryQ}
Note that $f$ does not have to contain the term
$x_1^{d_1}x_2^{d_2}\ldots x_n^{d_n}$ in this case, but the
powers occurring in the lexicographically largest
monomial of $f$ are 
at most $d_i$.

As a special case, with a uniform bound on the degrees and all sets
$S_i$ equal, we get:
\begin{corollaryQ}[Zippel 1979 {\cite[Theorem 1, p.~221]{z-pasp-79}}]
  \label{zippel}
  Suppose that $f$ is not identically zero and 
 the degree of each variable $x_i$ in $f$
is bounded by $d$, and
$S_1=S_2=\dots=S_n=S$.
Then the number of nonzeros is at least
$$
(|S|-d)^n=
|S|^n\cdot
\bigl(1-
{d}/{|S|}\bigr)^n .
$$
\end{corollaryQ}

The following statement puts a stronger assumption on $d$:
\begin{corollaryQ}
  [{DeMillo and Lipton 1978~\cite[Inequality~\thetag1]{dl-prapt-78}}]
  \label{deMillo}
If $f$ has total degree $d\ge0$ and
$S_1=S_2=\dots=S_n=S = \{1,2,\ldots,|S|\}$,
then the number of nonzeros is at least
$$
|S|^n\cdot
(1-d/|S|)^n.$$
\end{corollaryQ}
Note that this has essentially the same assumptions as
Corollary~\ref{schwartz-zippel} (only the assumption about the set $S$ is
more specialized),
but a weaker conclusion.

\subsection{Comparison between the results}
\label{sec:remarks}

The relation between the results in their published form
is confusing.
This is discussed at length
in \cite[Section~4]{bishnoi_2018}
and in several blog posts\footnote
{\url{https://anuragbishnoi.wordpress.com/2015/10/19/alon-furedi-schwartz-zippel-demillo-lipton-and-their-common-generalization/},
%
%
  \url{https://rjlipton.wpcomstaging.com/2009/11/30/the-curious-history-of-the-schwartz-zippel-lemma/}}.
Above, we have attempted to present them systematically in a logical
order, irrespective of the historic development.

As mentioned
in Section~\ref{sec:importance},
the precise bounds for the Qualitative Conclusion are of minor importance
for the 
applications, and
researchers may
 prefer to state their results in a form
that is more convenient to apply or easier to remember
instead of
the strongest form.
Thus, the reason that
Lemma~\ref{main-quantitative}, which is,
among the
statements with the Quantitative Conclusion,
 the strongest and most general so far, was apparently not written down before
is
simply that
nobody cared to do so.

\subsection{Precursor results}

We mention two precursor results:
In the first edition of
Knuth's \emph{Art of Computer Programming}, Vol.~2, 
there is 
a weaker, qualitative version of the Quantitative Conclusion:
\begin{corollaryQ}[Knuth 1969 {\cite[Ex.~4.6.1--16, p.~379, solution on
    p.~540\footnote{In the second edition, these are on
      p.~418 and p.~620.
In the third edition, this exercise has been replaced by the
statement of
Corollary~\ref{generalized-demillo-lipton-zippel}.}]{knuth2-1}}]
If $f$ is not identically zero and
$S_1=S_2=\dots=S_n=
\{-N,\allowbreak-N+1,\allowbreak
\ldots,\allowbreak N-1,N\}$,
then the fraction of zeros of $f$ in
$S_1\times S_2 \times \dots \times S_n$ goes to zero as $N\to \infty$.
%
\end{corollaryQ}

{\O}ystein Ore, in 1922, already established the special case of the
Schwartz--Zippel Lemma
(Corollary~\ref{schwartz-zippel})
 when the variables $x_i$ run over all elements of a finite field.
\begin{corollaryQ}
  [Ore 1922~\cite{ore22}, {\cite[Theorem 6.13]{Lidl-Niederreiter}}]
  If $f\in \mathbb F_q[x_1,\ldots,x_n]$
is a polynomial of total degree $d\ge0$
over a finite field $\mathbb F_q$ and
$S_1=S_2=\dots=S_n=\mathbb F_q$,
then the number of nonzeros is at least $(q-d)q^{n-1}$.  
\end{corollaryQ}
I have not been able to look are Ore's work, 
 and I am citing it
according to~\cite{Lidl-Niederreiter}.




\subsection{Proofs and extensions}
\label{sec:overview}

We give the very easy proofs of Lemmas~\ref{stark}
and \ref{main-quantitative}
in Sections
\ref{sec:stark}
and \ref{sec:quantitative},
respectively.
%
%
Another proof of Lemma \ref{stark}, which is based on
the technique of \emph{trimming} the polynomial,
is given in
Section~\ref{sec:trimming}.
It is the basis for the generalization of
Lemma \ref{stark} in
Section \ref{weaker-existence}.
Yet another proof of
Lemma~\ref{stark} is given
in  Appendix~\ref{proof-tao-vu}. 


In Section~\ref{sec:alon-fueredi},
we study the case where both the total degree and the individual degree
of each variable is constrained: This is
  the Generalized Alon--Füredi Theorem of~\cite{bishnoi_2018}.

The example in Section~\ref{sec:not}
shows that for a maximal
$x_1^{d_1}x_2^{d_2}\ldots x_n^{d_n}$,
the
Quantitative Conclusion in the form~\eqref{bound-equal}
does not follow.
In Section~\ref{sec:weak-conclusion} we explore the question what
quantitative statement we can nevertheless derive.
This question is wide open, and it
 leads to problems of extremal combinatorics and additive
 combinatorics.

There are many other
 extensions of
the Schwartz--Zippel Lemma or
 the Combinatorial Nullstellensatz.
 Among them,
 we mention
a ``multivariate'' generalization with a
quantitative conclusion~\cite{multivariate-SZ-22},
 giving an upper bound on the number of zeros of
$f$ over
$S_1\times S_2 \times \dots \times S_n$,
where the individual sets
$S_i\in K^{\lambda_i}$ are themselves multidimensional,
representing vectors or points or other geometric objects.
This is used to derive 
incidence bounds in combinatorial geometry.

\section{Proof of Lemma \ref{stark} by division by a linear factor}
\label{sec:stark}

We sketch the 
 proof of
Laso\'n~{\cite[Theorem 2]{l-gcn-10}}, which extends the
very simple proof of the original Combinatorial Nullstellensatz
(Corollary~\ref{alon}) that was given by
 Micha\l ek~\cite{michalek-AMM} in 2010.
\begin{proof}[Proof of Lemma \ref{stark}]
We use induction on $d_1+\cdots+d_n$.
The base case $d_1+\cdots+d_n=0$ is obvious. Otherwise, assume
w.l.o.g.\ that $d_1>0$. Pick an element $a\in S_1$ and divide $f$ by
$x_1-a$:
\begin{equation}
  \label{eq:divide}
  f = q(x_1-a)+r
\end{equation}
The remainder $r$ is of degree 0 in $x_1$, i.e., it is a function
$r(x_2,\ldots,x_n)$ and does not depend on $x_1$.
If $r$ has a nonzero on $S_2 \times \dots \times S_n$,
we obtain a nonzero of $f$ by setting $x_1=a$.
Suppose that
 $r$ is zero on all of $S_2 \times \dots \times S_n$. Then we get a nonzero
 of $f$
 by
 finding a nonzero of
 $q(x_1,x_2,\ldots,x_n)$ with $x_1\ne a$. The existence of such a nonzero
in $(S_1\setminus\{a\})\times S_2 \times \dots \times S_n$ 
 is ensured by the inductive hypothesis: It is easy to check that
$x_1^{d_1-1}x_2^{d_2}\ldots x_n^{d_n}$ is indeed a maximal
monomial of the quotient $q$.
\end{proof}

\section{Proof of Lemma \ref{main-quantitative}}
 \label{sec:quantitative}

\begin{proof}[Proof of Lemma \ref{main-quantitative}]

  The proof is by induction on~$n$.
The induction basis for $n=1$ is the elementary fact that a degree-$d$
polynomial has at most $d$ zeros.
For $n>1$,
  we write $f$ in powers of $x_1$:
\begin{equation}\label{lex-induction}
  f(x_1,\ldots,x_n) = \sum_{i=0}^{d_1} x_1^i h_i(x_2,\ldots,x_n)
\end{equation}
The sum contains in particular the nonzero term
$x_1^{d_1} h_{d_1}(x_2,\ldots,x_n)$.
By definition,
 $x_2^{d_2}\ldots x_n^{d_n}$ is the lexicographically largest monomial
 of $h_{d_1}$.
By induction, the number $N$ of tuples $(x_2,\ldots,x_n)\in S_2 \times
\dots \times S_n$
for which $h_{d_1}(x_2,\ldots,x_n)\ne 0$ is at least
$$N\ge 
(|S_2|-d_2)\cdots(|S_n|-d_n).$$
For a fixed $(x_2,\ldots,x_n)$ for which this case arises,
 $f$ is a 
 polynomial of degree 
 $d_1$ in $x_1$. Therefore it has at
most $d_1$ zeros, and at least $|S_1|-d_1$ nonzeros. Consequently,
the number of nonzeros of $f$ is
at least
\begin{displaymath}
  (|S_1|-d_1) N
  \ge   (|S_1|-d_1)(|S_2|-d_2)\cdots(|S_n|-d_n). \qedhere
\end{displaymath}
\end{proof}

\section{Largest total degree does not imply the Quantitative
  Conclusion}
\label{sec:not}

We show that maximality
(Lemma~\ref{main-existence})
and not even 
largest total degree
(Corollary~\ref{alon})
is not sufficient to derive
  the Quantitative Conclusion.
A counterexample is the polynomial $f(x_1,x_2)=x_1^2-x_1x_2+x_2^2-1$,
describing an ellipse in the plane,
and the sets $S_1=S_2=\{-1,0,1\}$,
see Figure~\ref{fig:nullstellensatz-counterexample}.
The monomial
$x_1x_2$ is a monomial
of largest total degree,
and
the
Quantitative Conclusion for $d_1=d_2=1$ would predict at
least
$(|S_1|-d_1)(|S_2|-d_2)=4$ nonzeros on $S_1\times S_2$. However, there
are only~$3$ nonzeros. (In fact, $3$ is the smallest possible number
of nonzeros 
for any polynomial for with $x_1x_2$ as
maximal monomial,
see Proposition~\ref{knuth-weak} in Section~\ref{sec:weak-conclusion}.) 

\begin{figure}[htb]
  \centering
  \includegraphics{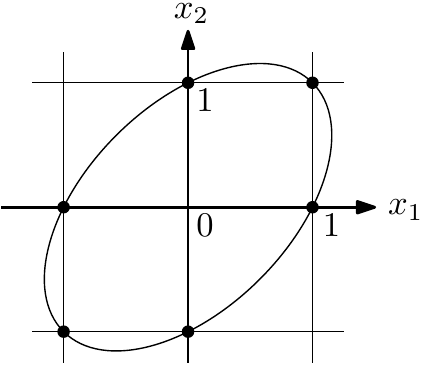}
  \caption{A quadratic bivariate polynomial with 6 zeros on a $3\times 3$ grid}
  \label{fig:nullstellensatz-counterexample}
\end{figure}

\section{Proof  of Lemma \ref{stark}
  by trimming}
\label{sec:trimming}

The Combinatorial Nullstellensatz
is a basic result,
and 
it appears in a wide range of textbooks.
Many of the proofs that I have seen in my
(not very thorough) survey of
the literature proceed
in two steps 
along the following lines.

The first step reduces the polynomial $f$ to a \emph{trimmed}
polynomial, 
whose degree \emph{in each variable} is now less than
$|S_i|$, without changing the value of~$f$ on
$S_1\times S_2 \times \dots \times S_n$;
After this reduction,
one can apply any of the lemmas with the
Quantitative Conclusion.

We include this proof because it lends itself to a generalization,
Theorem~\ref{schauz}
in Section~\ref{weaker-existence}.

The trimming procedure is described in the following statement:
\begin{proposition}
  \label{proposition-trimming}
Let
  $f\in K[x_1,\ldots,x_n]$ be a polynomial over a
  commutative ring $K$, and
  let $S_1,\ldots,S_n\subseteq K$ be 
  sets.

  Then $f$
  can be transformed into a polynomial $\hat f$ with the following
  properties:
  \begin{enumerate}
  \item $f$ and $\hat f$ have the same values on
     $S_1\times S_2 \times \dots \times S_n$.
   \item
     In $\hat f$, the
     degree {in each variable} $x_i$ is less than
$|S_i|$.
  \item If $x_1^{e_1}\ldots x_n^{e_n}$ is a maximal monomial of $f$
    with $e_i< |S_i|$ for all~$i$, then its coefficient remains
    unchanged by this transformation.
    \label{untouched}
  \end{enumerate}
\end{proposition}

\begin{proof}

  


Let $s_i=|S_i|$.
The polynomials
$x_i^{s_i}$ and
$x_i^{s_i}-\prod_{a\in S_i}(x_i-a)$
have the same values for all
  $x\in S_i$.
  Hence, we may successively replace
  $x_i^{s_i}$ by the polynomial
  $x_i^{s_i}-\prod_{a\in S_i}(x_i-a)$,
 whose degree is smaller than~$s_i$,
  and in this way, eliminate all
  powers of $x_i$ of degree $s_i$ or higher, without changing the value of $f$
  on $S_1\times S_2 \times \dots \times S_n$.
  (Putting it differently, we divide $f$ by $\prod_{a\in S_i}(x_i-a)$
  and take the remainder.)
  
  If we do this for all variables, we arrive at a polynomial
  $\tilde f$ for which the degree in each variable $x_i$ is less than
  $s_i$.

To see Property~\ref{untouched}, we observe that
 the modification, applied to a term
  $x_1^{e_1}x_2^{e_2}\ldots x_n^{e_n}$, only
  affects the coefficients of monomials
  $x_1^{b_1}x_2^{b_2}\ldots x_n^{b_n}$ with $b_i\le e_i$ for all $i$.
A monomial
$x_1^{e_1}x_2^{e_2}\ldots x_n^{e_n}$ with $e_i< s_i 
$ for all
  $i$
  is itself not subject to the trimming procedure, and
if it is  maximal, it has
  no monomials ``above it'' that could change its coefficient.
\end{proof}

Since the degree $d_i$ in each variable $x_i$ is now less than
$|S_i|$, we can apply
Corollary~\ref{generalized-demillo-lipton-zippel},
which has an easy inductive proof along the lines of the proof
of Lemma~\ref{main-quantitative} shown in Section~\ref{sec:quantitative},
or we may pick
a lexicographically
largest monomial and apply
Lemma~\ref{main-quantitative} directly.

\subsection{Comparison of the proofs}

It is instructive to compare the two proofs
of Lemma~\ref{stark} that
we have seen. 
The trimming procedure is essentially a polynomial division,
and it reduces the polynomial to a polynomial for which the
Quantitative Conclusion holds.
To prove the
Quantitative Conclusion, one applies induction on the number of variables,
as in the proof of
Lemma~\ref{main-quantitative} (Section~\ref{sec:quantitative}).
The 
induction step
is based on the fact
that a univariate polynomial of degree $d$ has at most $d$
roots.
This fact, finally, is proved by
repeated division by a linear factor.




By contrast, the proof of Section~\ref{sec:stark}, which goes back to
Micha\l ek~\cite{michalek-AMM}, puts the division by a linear factor
at the very beginning.  As we have seen, this makes the proof simple
and direct.

In  Appendix~\ref{proof-tao-vu}, we give another proof.
It follows the suggested hint for the solution of
Exercise~9.1.4
in Tao and Vu\cite[
p.~332]
{tv-ac-06},
and it is the earliest proof of
Lemma~\ref{stark}.
In contrast to the other proofs, it works only for fields.

%




\section{Weaker assumptions}
\label{weaker}

There is a way in which the respective assumptions of
Lemma~\ref{main-quantitative} and Lemma~\ref{main-existence} can be
weakened.  The two variations of the assumptions were developed
independently, but they are remarkably similar in spirit, and the
relation between them is analogous to the relation between
lexicographically largest and maximal monomials.  The assumptions
are
not easy to understand,
and they
are motivated mainly by the fact that the original proofs
carry through with few changes. 

\subsection{Successively largest 
  sequences for the Quantitative Conclusion}
\label{weaker-quantitative}

We define a more general notion than a lexicographically largest
monomial, namely what we call
 a \emph{successively largest sequence}
 $(d_1,\ldots,d_n)$ of exponents: 
%
Pick  \emph{any}
monomial $x_1^{e_1}x_2^{e_2}\ldots x_n^{e_n}$ of $f$.
We set
$f_1$ to be
the original polynomial
$f_1(x_1,\ldots,x_n)=f(x_1,\ldots,x_n)$.
For $j=2,\ldots, n$,
we inductively define $f_{j}(x_{j},\ldots,x_n)$
as the 
coefficient of
 $x_{j-1}^{e_{j-1}}$
 in
 $f_{j-1}(x_{j-1},\ldots,x_n)$.
 
Finally, we let 
$d_j$ be the degree of $x_j$ in $f_j$, for $j=1,\ldots,n$.



Consider, for example, the polynomial
$f(x_1,x_2)=x_1^7+x_1^6x_2^9+x_1x_2^2+x_1x_2+x_2^6$.
Picking the term $x_1x_2$ leads to $f_2(x_2)=x_2^2+x_2$, and thus
a {successively largest sequence}
$(d_1,d_2)=(7,2)$.
For the term $x_2^6$, 
we get 
$(d_1,d_2)=(7,6)$.
Figure~\ref{succ-largest} shows another example:
$(d_1,d_2)=(4,2)$ is
a {successively largest sequence} with respect to the monomial
$xy$. 

Note that $x_1^{d_1}x_2^{d_2}\ldots x_n^{d_n}$ is not necessarily a
monomial of $f$.
%
   As with the {lexicographically largest} monomial,
   this notion
 depends on the chosen order of the variables.

\begin{theorem}
  [Knuth 1998 
  {\cite[Answer to
    Ex.~4.6.1--16, pp.~674--675]{knuth2}}]
  \label{theorem-knuth}
For a \emph{successively largest sequence}
$d_1,\ldots,d_n$, the Quantitative Conclusion holds.  
\end{theorem}

\begin{proof}
The proof of Lemma~\ref{main-quantitative} goes through with
straightforward adaptations.
We proceed by induction on~$n$.
We write $f$ in
  powers of $x_1$ as in~\eqref{lex-induction}:
\begin{align*}
  f(x_1,\ldots,x_n) = \sum_{i=0}^{d_1} x_1^i h_i(x_2,\ldots,x_n)
\end{align*}
 By assumption,
 the sum contains 
 the nonzero term
$x_1^{e_1} f_2(x_2,\ldots,x_n)$.
By definition, $(d_2,\ldots,d_n)$ is a {successively largest sequence}
for $f_2$.

For a fixed tuple
$(x_2,\ldots,x_n)$ with 
$f_2(x_2,\ldots,x_n)\ne0$,
$f$ is a nonzero
polynomial of degree {at most} $d_1$ in $x_1$.
In contrast to the case of Lemma~\ref{main-quantitative},
 the degree can be smaller than $d_1$, but the conclusion that
 $f$ has hat most $d_1$~zeros remains valid.
The argument finishes in the same way as for Lemma~\ref{main-quantitative}. 
\end{proof}

Knuth 
\cite[
p.~675]{knuth2} mentions further ideas of
strengthening the bound, and points out the significance in the
context of sparse polynomials.

\subsection{Weaker assumptions for the Existence Conclusion}
\label{weaker-existence}
 
%

\begin{theorem}[Schauz 2008 {\cite[Theorem~3.2(ii)]{schauz-08}}]
  \label{schauz}
  Assume  $|S_i|>d_i\ge e_i$ for $i=1,\ldots,n$, and
 assume that
 $x_1^{e_1}\ldots x_n^{e_n}$ is a monomial of $f$.
If $f$
contains
no other monomial 
$x_1^{e'_1}\ldots x_n^{e'_n}$ with
$e'_i=e_i$ or $e'_i> d_i$ 
for each
$i=1,\ldots,n$, then the
Existence Conclusion holds. 
\end{theorem}
Figure~\ref{fig-d-leading} illustrates this condition.
In the terminology of Schauz,
the tuple %
$(e_1,\ldots,e_n)$ is called a
``$(d_1,\ldots,d_n)$-leading multi-index''.
The term
$x_1^{d_1}\ldots x_n^{d_n}$ is not required to appear in $f$. 

Theorem~\ref{schauz} may be stronger than Lemma~\ref{main-existence}.
For example, for the polynomial
\begin{displaymath}
  f(x_1,x_2)= x_1^4x_2^8 + x_1x_2 + x_1^6x_2^2,
\end{displaymath}
which is a sparser variant of the polynomial in
Figure~\ref{fig-d-leading},
we may take $(e_1,e_2)=(1,1)$ and  $(d_1,d_2)=(4,2)$.

The forbidden exponent pairs can be written concisely as
$\{e_1,d_1+1,d_1+2,d_1+3,\ldots\} \times
\{e_2,d_2+1,d_2+2,d_2+3,\ldots\}$,
except $(e_1,e_2)$ itself.

\begin{proof}[Proof of Theorem~\ref{schauz}]  
The proof by {trimming}
from Section~\ref{sec:trimming} goes through:
Observe that trimming a monomial
$x_1^{c_1}x_2^{c_2}\ldots x_n^{c_n}$ creates
 monomials in which the powers $x_i^{c_i}$ with $c_i<|S_i|$ are
 unchanged. Only the powers $x_i^{c_i}$ with $c_i\ge|S_i|$ are
 replaced by smaller powers.
Thus, the monomials $x_1^{e'_1}\ldots x_n^{e'_n}$ that are excluded by
the assumption of
Theorem~\ref{schauz} 
are precisely those
monomials 
whose
trimming process could affect
the chosen monomial $x_1^{e_1}\ldots x_n^{e_n}$.
%
\end{proof}

Schauz showed the stronger statement that the coefficient of
$x_1^{e_1}\ldots x_n^{e_n}$ can be represented in terms of
the values of $f$ on $S_1\times S_2 \times \dots \times S_n$,
thus generalizing the coefficient formula~\eqref{eq:F-tilde}
in Appendix~\ref{proof-tao-vu}.
For further information and more references, see \cite{clark-14}. 


\subsection{Connections between the assumptions}
\label{sec:connections}

There is a connection between Theorems
\ref{theorem-knuth} and~\ref{schauz}:
The assumptions of the first theorem imply
the assumptions of the second.
In particular,
  if
$(d_1,\ldots,d_n)$ is a successively largest degree sequence with
respect to the monomial
$x_1^{e_1}\ldots x_n^{e_n}$,
then
the assumptions of 
Theorem~\ref{schauz} hold.

Looking at the top two rows of
Figure~\ref{fig:conditions}, 
one can notice
some general pattern:
The conditions for the Quantitative Conclusion in the left column
(lexicographically largest monomial, 
successively largest sequence)
depend on the ordering of the variables,
whereas the conditions for the Existence Conclusion in the right column
(maximal monomial, 
the
$(d_1,\ldots,d_n)$-leading multi-index of
Theorem~\ref{schauz})
are insensitive to the variable order.

One can observe (and prove) the following curious connection
between the forbidden monomials, which are shown as shaded regions
of Figure~\ref{fig:conditions}:
The forbidden terms for
$x_1^{d_1}x_2^{d_2}\ldots x_n^{d_n}$
being
a maximal monomial
can be obtained as the
 intersection
 of the
  forbidden terms for 
  being
  a lexicographically largest monomial over all $n!$ orderings
  of the variables.

The same relation holds between
a
successively largest sequence
(Theorem~\ref{theorem-knuth})
and the condition 
of
Theorem~\ref{schauz},
if the defining monomial $x_1^{e_1}
\ldots x_n^{e_n}$ is held fixed.

\subsection{Applications of the generalized results}
\label{sec:applications}

In the applications of the 
Combinatorial Nullstellensatz or the Schwartz--Zippel Lemma and its
relatives,
the degree bounds on
the polynomial $f$ are derived a priori, and not by looking at a
particular polynomial that is explicitly given.
Thus, the added generality offered
by Theorems
\ref{theorem-knuth} and~\ref{schauz}
is only academic and of little practical use.
Even for the
Generalized Combinatorial Nullstellensatz
(Lemma~\ref{main-existence}),
we are not aware of a convincing 
application 
for which the classic
Combinatorial Nullstellensatz
(Corollary~\ref{alon}) would not suffice.

Such an application was indeed given by
Laso\'n~{\cite[Theorem 4]{l-gcn-10}}, but it
appears somewhat
fabricated.
The polynomial can be obtained
from some homogeneous polynomial $h(x_1,\dots,x_n)$
by
replacing each variable
$x_i$ by some polynomial $f_i(x_i)$
(and adding some linear terms).
In a homogeneous polynomial, every monomial is both maximal and of maximum
total degree, but after the modification,
the terms acquire different degrees,
and Corollary~\ref{alon} no longer applies.


\section{Stronger 
  constraints: The Generalized Alon--Füredi Theorem}
\label{sec:alon-fueredi}

 Bishnoi, Clark, Potukuchi, and Schmitt~
 \cite{bishnoi_2018}
 give a precise bound on the minimum number of
 nonzeros when,
in addition to a bound $d_i$ on the degree of each variable $x_i$, the total
degree $d$ is specified. The bound is not explicit:
It is formulated in terms of an optimization problem of
 minimizing the product of variables $y_i$ under linear
 constraints.
\begin{theorem}[The Generalized Alon--Füredi Theorem,
{Bishnoi et al.~
  \cite{bishnoi_2018}}]
\label{gen-AF}
Let $f$ be a polynomial
of total degree $d$,
whose degree in each variable $x_i$ is
at most~$d_i$, where $d_i<|S_i|$.
 Then $f$ has at least $N$~%
nonzeros
on $S_1\times S_2 \times \dots \times S_n$,
 where $N$ is
the optimum value of the following minimization problem:
 \begin{align}
   \mathrm{minimize \ \ }& y_1y_2\ldots y_n\label{product}\\
   \mathrm{subject\ to \ \ }& |S_i|-d_i\le y_i\le |S_i|,\ \mathrm{for}
\     i=1,\ldots,n \label{range}\\
&   \sum_{i=1}^n y_i = |S_1|+\cdots+|S_n|-d \label{sum}
 \end{align}
\end{theorem}

\begin{figure}[htb]
  \centering
  \includegraphics{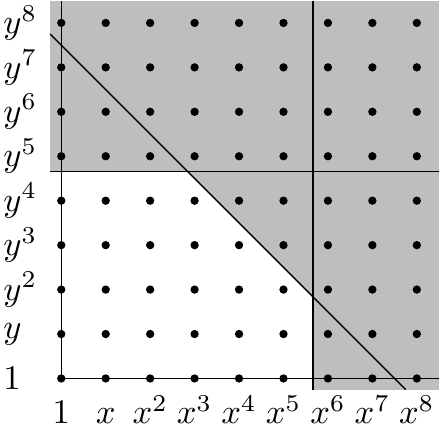}
  \caption{Forbidden monomials for the Generalized Alon--Füredi
    Theorem, for 
    $d_1=5, d_2=4, d=7$.
For an example with $|S_1|=|S_2|=8$,
the optimal value $N=y_1y_2=18$ is achieved by $(y_1,y_2)=(3,6)$.
}    
  \label{fig:alon-fueredi}
\end{figure}

Figure~\ref{fig:alon-fueredi} illustrates the assumptions. They
combine the constraints of Figure \ref{fig-partial-degree}
and~\ref{fig-total-degree}.

\begin{proof}
  The theorem 
  can be derived from
Lemma~\ref{main-quantitative}.
The optimization problem \thetag{\ref{product}--\ref{sum}}
can be interpreted as
looking for a lexicographically largest monomial
$x_1^{e_1}x_2^{e_2}\ldots x_n^{e_n}$ that is consistent with the
assumptions of the theorem and for which
Lemma~\ref{main-quantitative} gives the weakest bound.

To start the formal proof, note 
first that the optimum value $N$ of
\thetag{\ref{product}--\ref{sum}}
does not change if we turn \eqref{sum} into an inequality:
\begin{displaymath}
  \sum_{i=1}^n y_i \ge |S_1|+\cdots+|S_n|-d
  \leqno \thetag{\ref{sum}$'$}
\end{displaymath}
This is easily seen as follows:
Take a solution $(y_1,\dots,y_n)$ satisfying
\eqref{range} and~\thetag{\ref{sum}$'$}.
The assumptions of the theorem imply
$d\le \sum_{i=1}^n d_i$.
Therefore, as long as
the inequality \thetag{\ref{sum}$'$} is strict,
one can always find a variable $y_i$ that is not at its lower
bound, i.e., $y_i>|S_i|-d_i$. We can therefore reduce this variable,
reducing the product $y_1\ldots y_n$.

The proof is now straightforward:
Let
$x_1^{e_1}x_2^{e_2}\ldots x_n^{e_n}$ be the lexicographically largest
monomial of $f$. By the assumptions on~$f$, $e_i\le d_i$ and $\sum_{i=1}^n e_i \le
d$.
Hence, the quantities $y_i:= |S_i|-e_i$ satisfy the constraints~\eqref{range}
\begin{displaymath}
  |S_i|-d_i\le y_i\le |S_i|,
\end{displaymath}
and the constraint~\thetag{\ref{sum}$'$}:
\begin{displaymath}
  \sum y_i \ge |S_1|+\cdots+|S_n|-\sum e_i \ge |S_1|+\cdots+|S_n|-d
\end{displaymath}
By
Lemma~\ref{main-quantitative}, the number of nonzeros is at least
\begin{displaymath}
 (|S_1|-e_1)(|S_2|-e_2)\dots(|S_n|-e_n) = y_1y_2\ldots y_n, 
\end{displaymath}
which 
is at least the minimum value $N$ of
\thetag{\ref{product}} under
\eqref{range} and~\thetag{\ref{sum}$'$}.
\end{proof}

{Bishnoi et al.~
  \cite{bishnoi_2018}} proved Theorem~\ref{gen-AF} directly by induction on
$n$.
They showed that
the bound
is tight for all combinations of values $d$, $d_i$ and $|S_i|$ to which the
theorem applies.
They also derived 
the Generalized DeMillo--Lipton--Zippel Theorem
(Corollary~\ref{generalized-demillo-lipton-zippel}) from it.

In the (original) Alon--Füredi Theorem
\cite[Theorem~5]{ALON-FUEREDI-1993},
 the degrees $d_i$ in the individual variables are not
 constrained,
 and there is an 
 important difference: 
It is
\emph{assumed} that $f$ has at least one
nonzero on $S_1\times S_2 \times \dots \times S_n$.
Because of this 
assumption, the
Alon--Füredi Theorem is not a straightforward 
corollary of
the Generalized Alon--Füredi Theorem, see
\cite[Sections 2.2--2.3]{bishnoi_2018}.
In the constraints defining the bound $N$, the lower bound in
\eqref{range} is replaced by $y_i\ge 1$. 
As a consequence,
in contrast to
Theorem~\ref{gen-AF},
it is easy to solve the optimization problem: Starting from the lower
bound $y_1=\dots =y_n=1$, consider the variables $y_i$ in
order of decreasing sizes $|S_i|$ and
greedily 
enlarge each $y_i$ value to its upper bound $|S_i|$ until \eqref{sum} is fulfilled.

\section{Weaker quantitative conclusions for a maximal monomial}
\label{sec:weak-conclusion}

We have seen in Section~\ref{sec:not} that for
a maximal monomial, or even for
a monomial
of largest total degree, the
Quantitative Conclusion in the form~\eqref{bound-equal}
does not hold.   Can we still say something about
the number of nonzeros beyond the fact that it is at
least~1, which is the trivial consequence of the Existence Conclusion?

\subsection{Additive increase of the bound}
\label{sec:additive}

A very weak quantitative conclusion is given by the following statement. 
\begin{proposition}
\label{knuth-weak}
If
$x_1^{d_1}x_2^{d_2}\ldots x_n^{d_n}$ is a maximal monomial,
then
the number of nonzeros over the grid $S_1\times\dots\times S_n$,
with $|S_i|>d_i$ for all $i$, is at least
\begin{displaymath}
  1+ \bigl(|S_1|-(d_1+1)\bigr)+
  \bigl(|S_2|-(d_2+1)\bigr)+ \cdots +
  \bigl(|S_n|-(d_n+1)\bigr)
  .
\end{displaymath}
\end{proposition}

In other words, at each step of increasing
$|S_i|$ above the lower bound $d_i+1$ that is necessary for the
Existence Conclusion, the guaranteed number
of nonzeros increases by~1.

For example, with
$(d_1,d_2)=(1,1)$ and $|S_1|=|S_2|=3$, we conclude that there must be
at least 3 nonzeros.
Thus,
the ellipse example of Section~\ref{sec:not} cannot be improved by
choosing
a different grid $S_1\times S_2$ of the same size.

A version of
Proposition~\ref{knuth-weak}
 was stated
in 2022  by Knuth 
for the restricted case that
$x_1^{d_1}x_2^{d_2}\ldots x_n^{d_n}$ is a monomial of largest total
degree 
\cite[Ex.~MPR--114, p.~23, answer on p.~388]{knuth4B}.
 The proof goes through without changes when
$x_1^{d_1}x_2^{d_2}\ldots x_n^{d_n}$ is a maximal monomial and we base
the argument on
Lemma~\ref{main-existence}
instead of Corollary~\ref{alon}.

\begin{proof}[Proof of Proposition~\ref{knuth-weak}]
  We can eliminate any chosen nonzero $(x_1,\dots,x_n)$
  from $ S_1\times S_2 \times \dots \times S_n$
  by removing $x_j$ from $S_j$, for an arbitrary $j$.
  (This may eliminate additional nonzeros.)
  
  Thus, if there were fewer than the claimed number of nonzeros, we could
  eliminate them by successively removing an element from some $S_j$
  while keeping $|S_j|\ge d_j+1$.
Eventually we would arrive at a grid on which $f$ is identically zero, contradicting
Lemma~\ref{main-existence}.  
\end{proof}




\subsection{Hypergraph model}
Stronger asymptotic bounds can be obtained by using tools from
extremal combinatorics.
It is natural to associate an
$n$-partite $n$-uniform
hypergraph to the zeros 
of an $n$-variate polynomial
over a grid $S_1\times\dots\times S_n$:
The hypergraph contains the hyperedge
$(x_1,\ldots,x_n)$
whenever
$f(x_1,\ldots,x_n)=0$.
The Existence Conclusion then says
that the hypergraph contains no complete subhypergraph
$K^{(r)}(d_1+1,\dots,d_n+1)$.
%
What does this last statement alone (without regarding the
algebraic origin of the hypergraph) implies about
the number of nonzeros in
$S_1\times\dots\times S_n$. This is a question from extremal
(hyper-)graph theory.

We can apply the following result of
Erd\H os from 1964~\cite[Corollary, p.~188]{Erdoes1964OnEP}.

\begin{proposition}
  Consider the family of
  $n$-partite $n$-uniform hypergraphs
  that contain no
  complete $K^{(n)}(l,\dots,l)$, for some $l\ge 2$.
  
  Then there is a threshold $s_0(n,l)$ such that
  in every hypergraph of the family with at least $s$ vertices in each
  color class, for
  $s>s_0(n,l)$, the edge density is at most
  \begin{equation}
    \label{eq:hypergraph}
  (3n)^n \bigm/ 
  s^{1/l^{n-1}}
.    
  \end{equation}
\end{proposition}
(In the original statement in \cite{Erdoes1964OnEP}, our $n$ is denoted by $r$, which
adheres better to the conventions of hypergraphs, and our $s$ is
denoted by $n$.)



We translate this to our setting:
If $x_1^{d_1}x_2^{d_2}\ldots x_n^{d_n}$ is a maximal
monomial of~$f$,
Lemma~\ref{stark} implies
that the hypergraph corresponding
to the zeros does not
contain a complete $K^{(r)}(l,\dots,l)$, with
$l=1+\max \{d_1,\ldots, d_n\}$.
We conclude that the density of zeros
in $S_1\times S_2 \times \dots \times S_n$
is bounded
by~\eqref{eq:hypergraph}
if $s := \min \{|S_1|,\ldots, |S_n|\}$ is big enough.
This is good enough for the property that is essential for
the applications: 
The probability of hitting a zero goes to~0 as the size of all sets $S_i$ is
increased.
However, the
convergence is very slow.

\subsection{Bivariate polynomials}
\label{sec:bivariate}

For a polynomial of $n=2$ variables, we are in the setting of bipartite
\emph{graphs}, where the
classic result of K\H ov\'ari, S\'os, and Tur\'an~\cite{KST}
applies.
In particular, if 
 $x_1^{d_1}x_2^{d_2}$ is a maximal monomial, then
the bipartite graph with
$|S_1|+|S_2|$ vertices that models the zeros on $S_1\times S_2$ contains no complete bipartite subgraph
$K_{d_1+1,d_2+1}$.
Assuming $s=|S_1|=|S_2|$, we conclude
from the K\H ov\'ari--S\'os--Tur\'an Theorem
that
such a graph has at most $O(s^{2-1/l})$ edges, where
$l=\min\{d_1,d_2\}+1$.
Note that, in contrast to the case of hypergraphs above, we use
$\min\{d_1,d_2\}$ and not $\max$.
Hence the density of zeros is
\begin{displaymath}
  O(1/
\sqrt[l]s
  ).
\end{displaymath}
The bound of the
K\H ov\'ari--S\'os--Tur\'an Theorem is known to be tight for several small values
of $l$ in the combinatorial
setting, where all we know is that that
the bipartite subgraph
$K_{d_1+1,d_2+1}$ is forbidden.
This completely ignores the origin of the problem from the
polynomial~$f$. Can a polynomial with such a large fraction
$\Theta(1/s^{1/l})$ of zeros  on an $s\times s$ grid be constructed?

\subsection{A puzzle}

The first nontrivial example is
$(d_1,d_2)=(1,1)$, i.e., $xy$ 
should be a maximal monomial.
Such a polynomial, after suitable scaling, has the form
\begin{equation} \label{f-1-1}
  f(x,y)=-xy + P(x) + Q(y),
\end{equation}
where $P(x)$ and $Q(y)$ are polynomials of arbitrarily high degree.

Let us denote the elements that we substitute for $x$ by
$S_1=\{a_1,\ldots,a_s\}$, with distinct elements $a_i$,
and similarly for the values
$S_2=\{b_1,\ldots,b_s\}$ that we substitute for~$y$.
Let $c_i=P(a_i)$ and $d_j=Q(b_j)$ be the corresponding values of the
polynomials. Then the zeros of $f$
on $S_1\times S_2$ are the
index pairs $(i,j)$ with
  \begin{displaymath}
    a_ib_j = u_i+v_j \qquad (1\le i,j\le s).
  \end{displaymath}
We can thus reformulate our question as follows:

\begin{problem}
  \label{tables}
  Let $s$ be fixed.
  
Find two 
  sequences of 
  $a_1,\ldots, a_s$ and
  $b_1,\ldots, b_s$ of distinct numbers,
  and two
  sequences
  $u_1,\ldots, u_s$ and
  $v_1,\ldots, v_s$ of 
  not necessarily distinct numbers,
such that
  the multiplication table
  of the first two sequences agrees with the addition table of the last
  two sequences in as many positions $(i,j)$ as possible: 
  \begin{displaymath}
    a_ib_j = u_i+v_j
  \end{displaymath}
\end{problem}

For example,
the following multiplication and addition tables,
which are derived from the ellipse example of Section~\ref{sec:not},
 have 6 coinciding entries:

\begin{displaymath}
  \begin{tabular}{|r|rrr|}
    \hline
    $\times$&1&3&5\\
    \hline
    6&6&18&30\\
    7&7&21&35\\
    8&8&24&40\\\hline
  \end{tabular}
  \qquad  \text{ and }
  \qquad
  \begin{tabular}{|r|rrr|}
    \hline
    $+$&1&17&29\\
    \hline
    1&2&18&30\\
    6&7&23&35\\
    7&8&24&36\\
    \hline
  \end{tabular}
  \medskip
\end{displaymath}


The question has now become a problem of additive combinatorics.
It is clear that
Problem~\ref{tables} is 
not more restricted than asking for the zeros of~\eqref{f-1-1}:
We can find an interpolating polynomial $P$ and $Q$ for
any values $a_i$ and $c_i$, or $b_i$ and $d_i$, respectively,
since the degree of $P$ and $Q$ is not bounded.

As discussed above, the bipartite graph that models the zeros of~$f$
contains no $K_{2,2}$; this can also be shown directly from the
definition of an addition and multiplication table. Hence the number of
zeros is $O(s^{3/2})$.
Can this bound be achieved, asymptotically, or does the 
algebra
imply a sharper upper bound? Is there a
construction with a superlinear number of zeros?

\section{What's in a name?}

In the late 1970's, the first randomized primality tests
were discovered.
Randomized algorithms were gaining popularity,
and their usefulness 
was recognized.
It is thus no coincidence that various forms of
the Schwartz--Zippel Lemma were discovered independently,
as the topic was ``in the air''.
The papers of Schwartz and Zippel were even presented at the same
conference in 1979 and published back to back in the proceedings
volume \cite{s-pavpi-79,z-pasp-79}.

The name \emph{Schwartz--Zippel Lemma} stuck, despite
the accumulation of sibilant consonants, and
despite the priority
of DeMillo and Lipton~\cite{dl-prapt-78}.
A blog post of Richard Lipton\footnote
{
  \url{https://rjlipton.wpcomstaging.com/2009/11/30/the-curious-history-of-the-schwartz-zippel-lemma/}}
from 2009
proposed various possible reasons for this fact.
We add to this discussion by speculating that the
poor typesetting quality of the
\emph{Information Processing Letters} at the time may have
contributed to the fact that the paper~\cite{dl-prapt-78} was not sufficiently received.
In addition, the
quirk with the capital letter in the middle of the family name might have
caused some insecurity and uneasiness.
In the title of this note, we honor the tradition of omitting DeMillo
and Lipton.

 We have seen that Lasoń's generalization of Alon's Combinatorial
Nullstellensatz was predated by an exercise in a textbook, but he must
be nevertheless credited for bringing the 
statement of
Lemma~\ref{main-existence} to the published journal literature.
The major reason for including
his name is 
the rhyme.

\bibliographystyle{plainurl}
\bibliography{nullstellensatz}

\appendix
\section{
  Proof of Lemma \ref{stark} via the coefficient formula}
\label{proof-tao-vu}

This proof follows the hint of Tao and Vu~\cite[Exercise~9.1.4,
p.~332]{tv-ac-06}
and works out their exercise, see also
Laso\'n~{\cite[Section~3]{l-gcn-10}}.
Essentially the same proof, for 
the original
Combinatorial Nullstellensatz (Corollary~\ref{alon}), was given by
Kouba~\cite{kouba-09} in 2009.

As an intermediate result, we get a formula
\eqref{eq:F-tilde} for the coefficient of
$x_1^{d_1}x_2^{d_2}\ldots x_n^{d_n}$ 
in terms of the values of $f$ on
$S_1\times S_2 \times\dots \times S_n$
(
the Coefficient Formula of
Laso\'n~{\cite[Theorem 3]{l-gcn-10}}).


We emphasize, that in contrast to other
statements in this note, the following proof supposes that
the coefficient ring 
is a field (and we call it~$\mathbb F$).

We start with a preparatory lemma:

\begin{lemma}
 \label{interpol}
  Let $\mathbb F$ be a field.
  For a finite nonempty set $S\subseteq \mathbb F$, 
  there is a 
 function
 $
 g_{S}\colon S\to \mathbb F$ with the
following property:
\begin{align}
\label{null}
&  \sum_{x\in S} g_S(x)x^k = 0,\text{ for $k=0,1,\dots,|S|-2$}\\
\label{eins}
&  \sum_{x\in S} g_S(x)x^k = 1,\text{ for $k=|S|-1$}
\end{align}
\end{lemma}
\begin{proof}
  The equations \thetag{\ref{null}--\ref{eins}} form a 
  system
  of $|S|$ linear equations in the $|S|$ unknowns $u_j=g_S(a_j)$ for
  $a_j\in S= \{a_1,a_2,\dots,a_{|S|}\}$.  The coefficient matrix is
  a Vandermonde matrix, and hence the system has a unique
  solution.
  (The situation is the same as in Lagrange interpolation,
  except that the coefficient matrix is transposed.)
  
  The solutions $u_j$ can actually be obtained explicitly
  as
  the quotient of two
  Vandermonde determinants:
  \begin{equation}
   \label{eq:multiplier}
    u_j =
   g_S
  (a_j) = 1 \Bigm/ \prod_{k\ne j}(a_j-a_k) \qedhere
\end{equation}
\end{proof}

\begin{proof}[Proof of Lemma \ref{stark}]
It is no loss of generality to assume $|S_i|=d_i+1$.
Take the functions $
g_{S_i}$ for $i=1,\ldots,n$,
and multiply them together:
\begin{equation}
 \label{eq:multipliers}
  \tilde g(x_1,\ldots,x_n) :=
g_{S_1}(x_1) g_{S_2}(x_2) \ldots g_{S_n}(x_n)
\end{equation}
Continuing to follow the suggested procedure of Tao and
Vu~\cite[Exercise 9.1.4]{tv-ac-06}, 
we consider the quantity
\begin{equation}
  \label{eq:F-tilde}
\tilde F :=
\sum_{x_1\in S_1}
\sum_{x_2\in S_2}
\cdots
\sum_{x_n\in S_n}
f(x_1,\ldots,x_n) 
\tilde g(x_1,\ldots,x_n),
\end{equation}
and we want to show that $\tilde F\ne 0$.
Let us see how the transformation
from $f$ to $\tilde F$ affects the
monomials $x_1^{a_1}\ldots x_n^{a_n}$ of~$f$:
\begin{align}
\nonumber 
&
\sum_{x_1\in S_1}
\sum_{x_2\in S_2}
\cdots
\sum_{x_n\in S_n}
x_1^{a_1}\ldots x_n^{a_n}  
g_{S_1}(x_1) g_{S_2}(x_2) \ldots g_{S_n}(x_n)
\\
&\qquad\label{eq:sum-g2}
=
\sum_{x_1\in S_1}
x_1^{a_1}g_{S_1}(x_1)
\cdot
\sum_{x_2\in S_2}
x_2^{a_2}g_{S_2}(x_2)
\cdots
\sum_{x_n\in S_n}
x_n^{a_n} g_{S_n}(x_n)
\end{align}
This expression vanishes 
whenever $a_i<d_i$ for some $i$, by~\eqref{null}.
The only monomial of~$f$ that is not annihilated 
in this way
is the maximal monomial $x_1^{d_1}x_2^{d_2}\ldots x_n^{d_n}$. For this
monomial, the term~\eqref{eq:sum-g2} becomes 1,
 by~\eqref{eins}.
 Therefore $\tilde F$
as given by \eqref{eq:F-tilde}
is equal to the coefficient of
 $x_1^{d_1}x_2^{d_2}\ldots x_n^{d_n}$ in~$f$,
 expressing it
 in terms of the values of $f$ on the grid
 $S_1\times S_2 \times\dots \times S_n$. 
 Accordingly, \eqref{eq:F-tilde}, in connection with
\eqref{eq:multiplier} and \eqref{eq:multipliers},
 is called the \emph{coefficient formula}.

By the assumption of Lemma~\ref{stark},
$x_1^{d_1}x_2^{d_2}\ldots x_n^{d_n}$ appears in $f$, and thus
its coefficient
 $\tilde F \ne 0$.
Therefore, by \eqref{eq:F-tilde}, there must be an
 $(x_1,x_2,\dots,x_n)\in S_1\times S_2 \times\dots \times S_n$
with
$f(x_1,\ldots,x_n) \ne 0$.
\end{proof}

The hint of Tao and Vu~\cite[Exercise~9.1.4]{tv-ac-06} actually suggests to prove a more general version of
Lemma~\ref{interpol}:
\begin{lemma}
  \label{interpol1}
For a set $S$ with $|S|>d$, there is a 
 function
 $
 g_{S,d}\colon S\to \mathbb R$ with the
following property:
\begin{displaymath}
  \sum_{x\in S} g_{S,d}(x)x^k =
  \begin{cases}
  0,&\text{ for $k=0,1,\dots,d-1$}\\
  1,&\text{ for $k=d$}   
  \end{cases}
\end{displaymath}
\end{lemma}
This can be derived by applying
Lemma~\ref{interpol} to an arbitrary subset $S'\subseteq S$ of size
$|S'|=d+1$ and setting
$g_{S,d}(x)=0$ for $x\notin S'$.
We have instead chosen to simplify the proof by assuming 
$|S|=d+1$.
%

Tao and Vu~\cite[Exercise~9.1.4]{tv-ac-06} formulate their exercise
``for a field whose characteristic is 0 or greater than $\max d_i$.''
I don't see how the
characteristic of the 
field comes into play. 

Since we are constructing some sort of interpolating
function $g$, which depends on solving a system of equations,
this proof depends on $\mathbb F$ being a field (or at least,
a ring in which all nonzero differences $a-a'$ for $a, a'\in S_i$ are
units).
Under some weaker algebraic conditions (see
Section~\ref{sec:weaker-algebra}), it is still true that the
coefficient of $x_1^{d_1}x_2^{d_2}\ldots x_n^{d_n}$ in $f$ is uniquely
determined by the values of $f$
at the points
$(x_1,x_2,\dots,x_n)\in S_1\times S_2 \times\dots \times S_n$
\cite[Statement~2.8(v)]{schauz-08},
see also~\cite{clark-14}.

\end{document}